\documentclass[a4paper,12pt]{article}

\usepackage{a4,amssymb,amsmath,amsthm,url,xcolor,enumerate,float,lineno}
\usepackage{bezier,amsfonts,amssymb,graphicx,amsthm,url}
\usepackage[utf8]{inputenc}
\usepackage{subcaption}
\usepackage{algorithm}
\usepackage{algpseudocode}
\usepackage[color=lightgray!30]{todonotes}
\usepackage{amsmath}
\usepackage{amsfonts}
\usepackage{amssymb,amsthm}
\usepackage{url}
\usepackage{hyperref}
 \usepackage{tikzit}
\usepackage{xcolor,colortbl}
\newtheorem{theorem}{Theorem}

\newtheorem{proposition}[theorem]{Proposition}
\newtheorem{corollary}[theorem]{Corollary}
\newtheorem{lemma}[theorem]{Lemma}
\newtheorem{problem}[theorem]{Problem}
\newtheorem{conjecture}[theorem]{Conjecture}
\newtheorem{remark}[theorem]{Remark}

\tikzstyle{vertex}=[fill=black, draw=black, shape=circle, thick, scale=0.75]
\tikzstyle{blue_clique}=[fill=white, draw=blue, shape=circle, ultra thick]
\tikzstyle{clique}=[fill=white, draw=black, shape=circle, ultra thick]
\tikzstyle{red_vertex}=[fill=red, draw=red, shape=circle, scale=0.75, thick]

\tikzstyle{edge}=[-, fill={rgb,255: red,0; green,128; blue,128}, ultra thick]
\tikzstyle{dir_edge}=[->, thick]
\tikzstyle{blue_edge}=[-, draw=blue, ultra thick]
\tikzstyle{red_edge}=[-, draw=red, ultra thick]
\tikzstyle{box}=[-, fill=gray, fill opacity=0.2]
\tikzstyle{dashed_edge}=[-, dashed, thick]
\tikzstyle{dashed_red_edge}=[-, draw=red, dashed, thick]
\tikzstyle{dashed_blue_edge}=[-, draw=blue, dashed, thick]
\tikzstyle{dashed_edge_green}=[-, draw=green, dashed, thick]
\tikzstyle{green_edge}=[-, draw=green, ultra thick]

 \tikzset{every picture/.style={font issue=\footnotesize},
          font issue/.style={execute at begin picture={#1\selectfont}}
         }

\title{On zero-sum Ramsey numbers modulo 3}
\author{\begin{tabular}{cc}Yair Caro & Xandru Mifsud \\ University of Haifa-Oranim & University of Oxford \\ \href{mailto:yacaro@kvgeva.org.il}{\small yacaro@kvgeva.org.il} & \href{mailto:xandru.mifsud@cs.ox.ac.uk}{\small xandru.mifsud@cs.ox.ac.uk}\end{tabular}}
\date{}

\DeclareGraphicsExtensions{.pdf,.png,.jpg}

\begin{document}

%\linenumbers

\maketitle

\abstract{
We start with a systematic study of the zero-sum Ramsey numbers. For a graph $G$ with $0 \ (\!\!\!\!\mod 3)$ edges, the zero-sum Ramsey number is defined as the smallest positive integer $R(G, \mathbb{Z}_3)$ such that for every $n \geq  R(G, \mathbb{Z}_3)$ and every edge-colouring $f$ of $K_n$ using $\mathbb{Z}_3$, there is a zero-sum copy of $G$ in $K_n$ coloured by $f$, that is: $\sum_{e \in E(G)} f(e) \equiv 0 \ (\!\!\!\!\mod 3)$.

Only sporadic results are known for these Ramsey numbers, and we discover many new ones. In particular we prove that for every forest $F$ on $n$ vertices and with $0 \ (\!\!\!\!\mod 3)$ edges, $R(F, \mathbb{Z}_3) \leq  n+2$, and this bound is tight if all the vertices of $F$ have degrees $1 \ (\!\!\!\!\mod 3)$. We also determine exact values of $R(T, \mathbb{Z}_3)$ for infinite families of trees.
}

\vspace{1em}



\section{Introduction}

Given a graph $G$ and a positive integer $k$, a basic problem in Ramsey graph theory is the determination, or getting good approximations for, the smallest integer $R(G; k) = R(\underbrace{G, \dots, G}_{k \ \text{times}})$ such that every $k$ edge colouring of the complete graph with at least $R(G; k)$ vertices has a monochromatic copy of $G$. A major variant of this is the determination or estimation of the smallest integer $R(K_n, K_m)$ such that every bi-chromatic edge colouring of the complete graph with at least $R(n, m)$ vertices has a copy of $K_n$ in one colour, or a copy of $K_m$ in the other colour. 

Much efforts have also been invested in studying the Ramsey numbers of trees, with one of the main conjectures in this area being the following one, due to Burr and Erd\H{o}s.

\begin{conjecture}[Burr-Erd\H{o}s Conjecture \cite{BurrErdosConj}]
	Let $T$ be a tree on $n$ vertices. Then $R(T, T) \leq 2n - 2$ when $n$ is even, and $R(T, T) \leq 2n - 3$ otherwise.
\end{conjecture}

Alternatively, this conjecture could be restated as follows: For a tree $T$ on $n$ vertices, $R(T, T) \leq R(K_{1, n-1}, K_{1, n-1})$. Only in 2011 was a slightly weaker form of this conjecture proved by Zhao \cite{Zhao}, who showed that if $n$ is sufficiently large and $T$ is a tree on $n$ vertices then $R(T, T) \leq 2n - 2$. For a comprehensive reference on Ramsey numbers, see the ongoing survey by Radziszowski \cite{Radziszowski}, in particular Section 5 covering many results on trees. 

The basic paradigm of Ramsey graph theory is therefore concerned with the appearance of monochromatic `substructures' in ($k$-)edge-colourings of `big' structures. By contrast, \textit{zero-sum} Ramsey graph theory, started by Bialostocki and Dierker \cite{BialostockiDierker} in 1989, is concerned with colouring the edges of complete graphs (or other host graphs) with the elements of a given group (mostly the cyclic group of order $k$ denoted $\mathbb{Z}_k$), and the appearance of zero-sum substructures (instead of monochromatic). 

The main source of motivation was to obtain a graph-theoretic version of the celebrated Erd\H{o}s-Ginzburg-Ziv theorem, stating that every sequence of $2n - 1$ integers contains a subsequence of length $n$ the sum of whose elements is $0 \ (\!\!\!\!\mod n)$.

Formally, let $G$ be a graph with $e(G)$ edges and $k \geq 2$ a positive integer such that $k \vert e(G)$. We denote by $R(G, \mathbb{Z}_k)$ the minimum integer such that for $n \geq R(G, \mathbb{Z}_k)$ and any $f \colon E(K_n) \rightarrow \mathbb{Z}_k$, there is a copy of $G$, called a \textit{zero-sum copy of $G$}, such that $\sum_{e \in E(G)} f(e) \equiv 0 \ (\!\!\!\!\mod k)$. Henceforth, whenever speaking about $R(G, \mathbb{Z}_k)$ we shall assume that $k \vert e(G)$. For a reference concerning zero-sum Ramsey numbers, see the survey in \cite{CaroZS}. Naturally, this change of focus to zero-sum embeddings from monochromatic ones requires the use of algebraic tools such as the Erd\H{o}s-Ginzburg-Ziv theorem, the Chevalley-Warning theorem, the Baker-Schmidt theorem, and Davenport constants. 

The last years have seen zero-sum problems flourishing in many settings, including zero-sum problems over $\mathbb{Z}$ under certain labelling restrictions \cite{alvarado, Campbell, carohansberg, Caro2019, Caro2016, frick2023topological, Mohr2022, PARDEY2023108, domingo2023deep}. Among the main results in this area is the complete determination of $R(G, \mathbb{Z}_2)$ for all graphs, first in \cite{CARO1994205} and then reproved using distinct methods in \cite{CARO19991} and \cite{WILSON2014490}.

\begin{theorem}[\cite{CARO1994205}]
	Let $G$ be a simple graph on $n$ vertices and an even number of edges. Then,
	\[
		R(G, \mathbb{Z}_2) = 
		\begin{cases}
			n + 2,& \text{if } G \text{ is the complete graph} \\
			n + 1,& \text{if } G \text{ is the union of two complete graphs} \\
			& \text{or } G \text{ is not } K_n \text{ and has all vertices of odd degree}\\
			n,& \text{otherwise.}
		\end{cases}
	\]
\end{theorem}

This serves as motivation for the current paper, namely to start a more systematic exploration of $R(G, \mathbb{Z}_3)$, as the next most simple group. 

\subsection{Determining $R(G, \mathbb{Z}_3)$}

Throughout the years there has been a consistent effort concerning the determination of $R(G, \mathbb{Z}_3)$ \cite{NogaCaro, CaroZSCompleteG, CaroZS, CaroZSBinomCoeff, ChungGraham1, HarborthPiepmeyer, HarborthPiepmeyer2}, mostly concerning $R(K_n, \mathbb{Z}_3)$, which is still to this date not yet completely solved, leaving the following problem still open:

\begin{problem}
Determine for $n \equiv 	7 \ (\!\!\!\!\mod 9)$ the exact value of $R(K_n, \mathbb{Z}_3)$, which is either $n + 3$ or $n + 4$.
\end{problem}

The second problem we mention here is the following one:

\begin{problem}
Let $G$ be a graph $n$ vertices such that $3 \vert e(G)$. Is $R(G, \mathbb{Z}_3) \leq n + 8$ with equality if, and only if, $G = K_3$?	
\end{problem}

In this paper we shall mostly consider $R(T, \mathbb{Z}_3)$, where $T$ is a tree on $n \equiv 1 \ (\!\!\!\!\mod 3)$ vertices so that $3 \vert e(T)$. For small trees on 4 and 7 vertices, several cases have been solved throughout the years \cite{CaroZS}. Amongst the many open problems remaining concerning trees, we mention the following conjecture.

\begin{conjecture}[\cite{CaroRoditty}]\label{conj_tree_gen}
	Let $m \geq k \geq 2$ be integers such that $k \vert m$ and let $T_{m+1}$ be a tree on $n = m + 1$ vertices. If $G$ is a graph with minimum degree at least $n + k - 2$, then every $\mathbb{Z}_k$-edge-colouring of $G$ forces a zero-sum (modulo $k$) copy of $T_n$. 
\end{conjecture}

It is easy to check that Conjecture \ref{conj_tree_gen} implies that $R(T_n, \mathbb{Z}_k) \leq n + k -1$ and in particular, relevant to our case, that $R(T_n, \mathbb{Z}_3) \leq n + 2$. This motivated the following further conjecture by Caro \cite{caro2019problem}. 

\begin{conjecture}[\cite{caro2019problem}] \label{conj_tree}
	Let $T$ be a tree on $n \equiv 1 \ (\!\!\!\!\mod 3)$ vertices. Then,
	\[
		R(T, \mathbb{Z}_3) = 
		\begin{cases}
			n + 2,& \text{if } T \text{ is the star } K_{1, n-1} \\
			n + 1,& \text{if } T \text{ has no vertex of degree } 0 \ (\!\!\!\!\!\!\mod 3)\\
			n,& \text{otherwise.}
		\end{cases}
	\]
\end{conjecture}

The exact value of $R(K_{1, n-1}, \mathbb{Z}_3)$ is already known to be $n + 2$ \cite{CARO19921}.

In Section 2 we prove an upper-bound for $R(G, \mathbb{Z}_3) \leq n + 2$ using the notion of $2$-good graphs. In particular, we show that for every acyclic graph $F$ on $n$ vertices and $0 \ (\!\!\!\!\mod 3)$ edges, $R(F, \mathbb{Z}_3) \leq n+2$ and this bound is sharp if all vertices of $F$ have degree $1 \ (\!\!\!\!\mod 3)$.
  
In Section 3 we consider restrictive colourings, which are those colourings in which every vertex in $K_n$ is incident with at least $n-2$ edges with the same colour. We then prove a packing result for trees which are not a star in dense graphs.

In Section 4 we focus on Conjecture \ref{conj_tree}, by considering $R(T, \mathbb{Z}_3)$ for trees having induced subtrees with certain automorphism groups.  In conjunction with the tools developed in Section 3, we prove that for such trees $R(T, \mathbb{Z}_3) \leq n+1$ and in several cases we fully determine $R(T, \mathbb{Z}_3)$.

We conclude this paper by mentioning some further directions for research in Section 5.

\section{Upper-bounds on $R(G, \mathbb{Z}_3)$}

We begin by establishing the upper-bound $R(T, \mathbb{Z}_3) \leq n + 2$; we shall do so in a more general context. A graph $G$ (not necessarily connected) is said to be \textit{$2$-good} if it has at least $2$ vertices of degree 1 that are pairwise a distance at least three apart. Observe that every tree on $n \geq 3$ vertices which is not isomorphic to the star $K_{1, n-1}$ is $2$-good. Other examples of $2$-good graphs, distinct from a tree, are the complete graph $K_3$ with a leaf attached to every vertex, and the forest $t K_2$ for $t \geq 2$.

\begin{theorem} \label{r_g_z3_upper}
If $G$ be a $2$-good graph on $n$ vertices with $3 \vert e(G)$, then $R(G, \mathbb{Z}_3) \leq n + 2$.
\end{theorem}

\begin{proof}
Since $G$ is a $2$-good graph, there are two distinct vertices $u$ and $v$ adjacent to a leaf, say $w$ and $z$ respectively such that $\deg(w) = \deg(z) = 1$. Consider an edge-colouring $f$ of $K_{n+2}$ using $\mathbb{Z}_3$. For brevity, we shall abuse notation slightly and write $f(a, b)$ to mean $f(\{a, b\})$.

\textit{\textbf{Case 1:}} Suppose there is a vertex $h$ in $K_{n+2}$ incident with three edges ($\{h, a\}$, $\{h, b\}$ and $\{h, c\}$), each having distinct values under $f$. Consider the $(n-1)$-clique induced by the vertices distinct from $a, b$ and $c$, and embed of a copy of $G^\ast = G \backslash \{w\}$ in this clique with $u$ identified with $h$.

No matter what the weight of this embedding of $G^\ast$ is, call it $w(G^\ast)$, we can extend this embedding to a zero-sum embedding of $G$ in $K_{n+2}$ by adding the suitable edge from $\{h, a\}$, $\{h, b\}$ and $\{h, c\}$ having weight $-w(G^\ast)$, and identifying it with the edge $\{u, w\}$.

\textit{\textbf{Case 2:}} Suppose there is a vertex $h$ in $K_{n+2}$ incident with two edges ($\{h, a\}$ and $\{h, b\}$), each having distinct values under $f$. Consider the $(n-1)$-clique induced by the vertices distinct from $a, b$ and $h$.

\textit{\textbf{Case 2.1:}} Suppose that in the $(n-1)$-clique induced by the vertices distinct from $a, b$ and $h$, there is a vertex $q$ with two incident edges ($\{q, s\}$ and $\{q, t\}$), each having distinct values under $f$.

Then consider the $(n-2)$-clique induced by the vertices distinct from $a, b, s$ and $t$, and embed a copy of $G^{\ast\ast} = G \backslash \{w, z\}$ in this clique such that $u$ and $v$ are identified with $h$ and $q$ respectively.

By the simplest form of the Cauchy-Davenport theorem for $\mathbb{Z}_3$, since $A = \{f(h, a), f(h, b)\}$ and $B = \{f(q, s), f(q, t)\}$ are two sets containing two distinct elements of $\mathbb{Z}_3$, then $|A + B| = 3$.
 
Therefore, no matter what the weight of this embedding of $G^{\ast\ast}$ is, call it $w(G^{\ast\ast})$, we can extend this embedding to a zero-sum embedding of $G$ in $K_{n+2}$ by adding a suitable pair of edges whose total weight is $-w(G^{\ast\ast})$, namely: an from $\{h, a\}$ and $\{h, b\}$ identified with $\{u, w\}$, and a suitable edge from $\{q, s\}$ and $\{q, t\}$ identified with $\{v, z\}$.

\textit{\textbf{Case 2.2:}} Suppose that in the $(n-1)$-clique induced by the vertices distinct from $a, b$ and $h$, there is no vertex with two incident edges coloured distinctly; then clearly this clique must be monochromatic. 

Consider any vertex $x$ in this clique, and the path $ahb$ (in this order) in $K_{n+2}$ along with the edge $\{x, h\}$. If $f(x, h)$ has a value distinct from the values of $f(h, a)$ and $f(h, b)$, then $h$ is a vertex incident with three distinctly coloured edges, and by \textit{Case 1} we are done.

Hence, without loss of generality, we may assume that $f(x, h) = f(h, a)$. Since $f(h, a) \neq f(h, b)$, we also have that $f(x, h) \neq f(h, b)$. 

Consider the $(n-1)$-clique induced by the vertices distinct from $x, h$ and $b$. If in this clique there is a vertex $q$ as in \textit{Case 2.2}, then we are done. 

Otherwise consider the case when this clique is also monochromatic. This means that in the $n$-clique induced by the vertices distinct from $h$ and $b$, all the edges except $\{x, a\}$ may be monochromatically coloured. Clearly there is a monochromatic of $G$ in this $n$-clique, and we are done.

\textit{\textbf{Case 3}}: Lastly, suppose that there is no vertex $h$ in $K_{n+2}$ incident to two distinctly coloured edges. Then $K_{n+2}$ is monochromatic and any embedding of $G$ is zero-sum (modulo 3).
\end{proof} 

As an immediate consequence of the above, we get the following upper-bound for trees.

\begin{corollary}
	If $T$ is a tree on $n \equiv 1 \ (\!\!\!\!\mod 3)$ vertices, then $R(T, \mathbb{Z}_3) \leq n +2$.
\end{corollary}

We next note the following useful lemma giving us lower-bounds on $R(G, \mathbb{Z}_3)$ for graphs whose degree-sequence satisfies certain properties.

\begin{lemma} \label{r_g_z3_lower}
Let $G$ be a graph on $n$ vertices with $3 \vert e(G)$.

\begin{enumerate}[i.]
	\item If every vertex of $G$ has degree $1 \ (\!\!\!\!\mod 3)$, then $R(G, \mathbb{Z}_3) \geq n + 2$.
	\item If no vertex of $G$ has degree $0 \ (\!\!\!\!\mod 3)$, then $R(G, \mathbb{Z}_3) \geq n + 1$.
\end{enumerate}	
\end{lemma}

\begin{proof}
We shall first consider the case no every vertex of $G$ has degree $0 \ (\!\!\!\!\mod 3)$. Consider $K_n$ and choose a vertex $v$. Colour all edges incident to $v$ using $1$, and colour the remaining edges using $0$. In any embedding of $G$, a vertex $w$ of $G$ must be located on $v$. So the weight of this embedding is just $\deg(w)$, but as $\deg(w)$ is not $0 \ (\!\!\!\!\mod 3)$, we have a colouring of $K_n$ such that no embedding of $G$ has zero-sum. Therefore $R(G, \mathbb{Z}_3) \geq n + 1$. 

We next prove the case when every vertex of $G$ has degree $1 \ (\!\!\!\!\mod 3)$. Consider $K_{n+1}$ and choose two vertices $x$ and $y$. Colour all the edges incident to $x$ and/or $y$ using $1$, and colour the remaining edges using $0$. 

Clearly, any embedding of $G$ in $K_{n+1}$ must use at least one of $x$ or $y$. In case an embedding of $G$ uses only $x$ (respectively $y$), then such an embedding has weight $1$ since $x$ has degree $1 \ (\!\!\!\!\mod 3)$ and all edges incident to it are coloured $1$ and the remainder is coloured $0$.

Otherwise consider any embedding of $G$ that uses both $x$ and $y$. If the edge $\{x, y\}$ is not used by the embedding, then the weight of the embedding must by $2 \ (\!\!\!\!\mod 3)$. Otherwise if the edge $\{x, y\}$ is used by the embedding, then the weight of the embedding must by $1 \ (\!\!\!\!\mod 3)$. 

Hence we have a colouring of $K_{n+1}$ such that no embedding of $G$ has zero-sum. Therefore, $R(G, \mathbb{Z}_3) \geq n + 2$.
\end{proof}

\begin{corollary} \label{2good_sharp}
	Let $G$ be a $2$-good graph on $n$ vertices with $3 \vert e(G)$ and all degrees $1 \ (\!\!\!\!\mod 3)$. Then $R(G, \mathbb{Z}_3) = n + 2$.
\end{corollary}

We take a moment to note the following observations. Firstly note that no tree $T$ with $3 \vert e(T)$ and all vertex degrees $1 \ (\!\!\!\!\mod 3)$ exists, as a consequence of the hand-shaking lemma. However, while no such trees exist, such forests are very common, leading us to the following consequence.

\begin{corollary}
	If $F$ is an acyclic graph on $n$ vertices with $3 \vert e(F)$, then $R(F, \mathbb{Z}_3) \leq n + 2$, and this bound is sharp if all vertex degrees are $1 \ (\!\!\!\!\mod 3)$.
\end{corollary}

\begin{proof}
Observe that there are three possibilities for $F$: $F$ is either $2$-good, a star $K_{1, n-1}$, or a star $K_{1, m-1}$ and $n - m$ isolated vertices for some $m < n$ such that $m \equiv 1 \ (\!\!\!\!\mod 3)$. 

If $F$ is $2$-good then we are done by Theorem \ref{r_g_z3_upper}, and if all the vertex degrees are $1 \ (\!\!\!\!\mod 3)$ this bound is sharp by Corollary \ref{2good_sharp}. If $F$ is the star $K_{1, n-1}$ then it is known that $R(F, \mathbb{Z}_3) = n+2$. Otherwise if $F$ is the star $K_{1, m-1}$ with $n - m$ isolated vertices, then $R(F, \mathbb{Z}_3) \leq n + 2$, namely by considering a zero-sum embedding of the star in an $(m+2)$-clique of $K_{n+2}$, and then mapping the isolated vertices to the remainder. 
\end{proof} 

\section{Restrictive edge-colourings of $K_n$}

An edge-colouring of $K_n$ is said to be \textit{restrictive} if every vertex has at least $n - 2$ incident edges coloured the same. For such a colouring, the uniquely coloured edge incident to a vertex (if it exists) is called a \textit{bad edge}. The following lemma characterises the restrictive colourings of $K_n$ when using at most three colours.

\begin{lemma} \label{bad_edge_removal}
Let $n \in \mathbb{N}$ such that $n \geq 4$ and consider a restrictive edge-colouring of $K_n$ using at most three colours. Then the coloured subgraph obtained after deleting the bad edges from the colouring is a monochromatic graph which is either:
\begin{enumerate}[i.]
	\item connected on $n$ vertices with minimum degree at least $n - 2$,
	\item or the union of $K_1$ and $K_{n-1}$. 
\end{enumerate}	
\end{lemma}

\begin{proof}
Firstly, if every vertex has all incident edges coloured the same, then $K_n$ is monochromatic: Considering the star $K_{1,n-1}$ centred at a vertex $v$, which must be monochromatic. If $u$ is a leaf of this star, the star at $u$ must have the same colour, as otherwise $\{u, v\}$ is a `bad' edge for $u$. Hence $K_n$ must be monochromatic.

Otherwise, suppose some vertex $v$ has $n-2$ blue edges and $1$ red edge, say between $v$ and $u$. Let $V'$ be the remaining $n - 2$ vertices.

\textit{\textbf{Case 1:}} If all $n-1$ edges incident to $u$ are red, $V'$ must form a monochromatic red or blue clique (otherwise there is some vertex which has either all three colours incident to it, or two red edges and two blue edges). If the clique is red, all edges incident to $v$ are bad for some vertex, and deleting them leaves a red $K_1 \cup K_{n-1}$. If the clique is blue, all edges incident to $u$ are bad for some vertex, and removing them leaves a blue $K_1 \cup K_{n-1}$.

\textit{\textbf{Case 2:}} Suppose $u$ has a uniquely coloured edge.

\textit{\textbf{Case 2.1:}} If this edge is the red edge $\{u, v\}$, the remaining $n-2$ edges incident to $u$ are either all blue or all green. If blue, the bad edges among $V'$ must be red or green and form matchings (otherwise there is some vertex which has either all three colours incident to it, or two red edges and two green edges). Removing these along with $\{u, v\}$ leaves a blue graph on $n$ vertices with minimum degree $n-2$, forcing connectivity. If green, $V'$ induces a monochromatic blue or green clique (by an argument similar to \textit{Case 1}), and all the bad edges are incident to $u$, which when removed leaves us with a blue $K_1 \cup K_{n-1}$.

\textit{\textbf{Case 2.2:}} Suppose the uniquely coloured edge incident to $u$ is different from the red edge $\{u, v\}$. Then $u$ has $n-2$ red edges and one edge $\{u, w\}$, where $w \in V'$, which is either blue or green.

If $\{u, w\}$ is green, $w$ has a blue edge (via $v$) and a green edge (via $u$). All other edges incident to $w$ must be blue or green to avoid three colours at $w$. For any $x \in V'$ with $x \neq w$, $x$ has a blue edge (via $v$) and a red edge (via $u$), so $\{w, x\}$ must also be blue. Thus, all edges incident to $w$ except $\{u, w\}$ are blue, and removing the bad edges incident to $u$ leaves a blue $K_1 \cup K_{n-1}$.

If $\{u, w\}$ is blue, no vertex in $V'$ can have a green edge, as that would force three colours at some vertex. Moreover, $w$ can have at most one red edge. If $\{w, x\}$ is red for some $x \in V'$, all other vertices in $V'$ must have a red edge incident to $x$. This creates vertices with two red edges (via $u$ and $x$) and two blue edges (via $v$ and $w$), which cannot be the case. Thus, all edges incident to $w$ are blue, and removing the bad edges incident to $u$ leaves a blue $K_1 \cup K_{n-1}$. The result follows.
\end{proof}

The following sequence of lemmas establishes that for a restrictive colouring of $K_n$ using at most three colours, if $T$ is a tree on $n$ vertices which is not isomorphic to the star $K_{1, n-1}$, then $T$ is embeddable into the graph $G$ obtained from $K_n$ by removing all of the bad edges from the colouring of $K_n$. 

\begin{lemma} \label{tree_leaf_lemma}
Let $n \in \mathbb{N}$ such that $n \geq 6$. If $T$ is a tree on $n$ vertices which is not isomorphic to $K_{1, n-1}$, then there exists two leaves $u$ and $v$ such that $T \backslash \{u, v\}$ is not isomorphic to $K_{1, n-3}$. 
\end{lemma}

\begin{proof}
	Let $u$ and $v$ be leaves on the longest path in $T$. Since $T$ is not star, then $d(u, v) \geq 3$. Let $w$ and $z$ be the neighbours of $v$ and $u$, respectively, in the path from $u$ to $v$. If $w$ is adjacent to at least three leaves remove two of them and the diameter of the obtained tree remains at least $3$ hence is not a star. The same holds for $z$. Else, if both $w$ and $z$ are adjacent to two leaves, remove one leaf from each of them. The obtained tree has diameter at least $3$ and hence is not a star.
 
	If the total number of leaves incident to both $w$ and $z$ is three, then $d(u, v) \geq 4$ as otherwise $T$ has at most 5 vertices. Hence we can remove a leaf adjacent to $w$ and a leaf adjacent to $z$ and the obtained tree is not a star. Otherwise if the total number of leaves incident to $w$ and $z$ is $2$, then $d(u, v) \geq 4$ and either there is a vertex $x$ of degree at least 3 on the path between $w$ and $z$ which is either adjacent to a leaf $y$ or is the branching vertex of a subtree having another leaf $y$ which can be removed together with the leaf adjacent to $w$ to get the required tree. Else the diameter of $T$ is at least 5, otherwise $T$ is the path $P_5$ on 5 vertices, and we can remove the two leaves adjacent to $w$ and $z$ to obtain a tree with diameter at least 3 (which is not a star). 
\end{proof}

\begin{lemma} \label{tree_embedding_lemma}
	Let $n \in \mathbb{N}$ such that $n \geq 6$. Let $G$ be a graph on $n$ vertices with minimum degree at least $n - 2$. If $T$ is a tree on $n$ vertices not isomorphic to $K_{1, n-1}$, then $T$ embeds in $G$. 
\end{lemma}

\begin{proof}
	We shall proceed by induction on $n$. For $n = 6$, there are four possible graphs with minimum degree at least $4$ and five trees not isomorphic to $K_{1, 6}$. It is easily checkable that each of these five trees can be embedded in each of these four graphs. Hence the base case holds. 
	
	Suppose the result holds for all $k \in \mathbb{N}$ such that $6 \leq k < n$. Let $G$ be a graph on $n$ vertices with minimum degree at least $n - 2$ and let $T$ be a tree on $n$ vertices which is not isomorphic to $K_{1, n-1}$. 
	
	Suppose that $G$ has a vertex $x$ of degree $n - 1$ and consider $G - x$. Then $G - x$ is a graph on $n - 1 \geq 6$ vertices (since $n \geq 7$) with minimum degree at least $n - 3$. Since $T$ is not isomorphic to $K_{1, n-1}$ then there exists a leaf $u$ in $T$ such that $T - u$ is not isomorphic to $K_{1, n-2}$. Then by the inductive hypothesis, $T - u$ embeds in $G - x$. Since $x$ is incident to all other vertices, mapping the embedding of $T - u$ in $G - x$ to $G$, we can get an embedding for $T$ in $G$ by identifying $u$ with $x$.
	
	Otherwise consider the case when all vertices in $G$ have degree $n - 2$. Then $G$ must be the complete graph with a perfect matching removed, and hence $n$ must be even. In particular, $n \geq 8$. Let $x$ and $y$ be two vertices in $G$ such that $x$ is not adjacent to $y$. Then $x$ and $y$ are both adjacent to all the vertices in $G \backslash \{x, y\}$. Now, $G \backslash \{x, y\}$ is a graph on $n - 2 \geq 6$ vertices with minimum degree at least $n - 4$. By Lemma \ref{tree_leaf_lemma}, there exists two leaves $u$ and $w$ in $T$ such that $T \backslash \{u, w\}$ is a subtree of $T$ on $n - 2$ vertices which is not isomorphic to $K_{1, n-3}$. Hence by the inductive hypothesis, $T \backslash \{u, w\}$ embeds in $G \backslash \{x, y\}$. Since $x$ and $y$ are incident to all other vertices, mapping the embedding of $T \backslash \{u, w\}$ in $G \backslash \{x, y\}$ to $G$, we can get an embedding for $T$ in $G$ by identifying $u$ and $w$ with $x$ and $y$. 
	
	Hence the inductive case holds and the result follows.
\end{proof}

\begin{corollary} \label{tree_embedding}
	Let $n \in \mathbb{N}$ such that $n \geq 6$. Let $G$ be a graph with a component on $n$ vertices with minimum degree at least $n - 2$. If $T$ is a tree on at most $n$ vertices not isomorphic to $K_{1, n-1}$, then $T$ embeds in $G$.
\end{corollary}

\begin{proof}
	If $T$ has $n$ vertices and is not $K_{1, n-1}$, this follows immediately from Lemma \ref{tree_embedding_lemma}. Otherwise, it follows from the well-known fact that every graph with minimum degree at least $n-2$ contains all trees on at most $n - 1$ vertices.
\end{proof}

As a consequence of the above, we have the following useful proposition for zero-sum embeddings of trees in restrictive colourings.

\begin{proposition} \label{zero_sum_restrictive_colouring_1}
	Let $n \in \mathbb{N}$ such that $n \geq 6$ and $n  \equiv 1 \ (\!\!\!\!\mod 3)$. Let $T$ be a tree on $n$ vertices, not isomorphic to $K_{1, n-1}$. If $T$ has a vertex of degree $0 \ (\!\!\!\!\mod 3)$, there is a zero-sum embedding of $T$ in every restrictive colouring of $K_n$ using $\mathbb{Z}_3$. 
	
	Otherwise if $T$ does not have a vertex of degree $0 \ (\!\!\!\!\mod 3)$, there is a zero-sum embedding of $T$ in every restrictive colouring of $K_{n+1}$ using $\mathbb{Z}_3$.
\end{proposition}

\begin{proof}
	Consider a restrictive colouring of $K_n$ using $\mathbb{Z}_3$. By Lemma \ref{bad_edge_removal}, the monochromatic graph $G$ obtained by deleting the uniquely coloured edges incident to each vertex is either a connected graph on $n$ vertices with minimum degree at least $n - 2$, or $K_1 \cup K_{n-1}$.

    In the case that $G$ is connected on $n$ vertices with minimum degree at least $n-2$, since $T$ is not isomorphic to $K_{1, n-1}$ then by Corollary \ref{tree_embedding} we have that $T$ embeds in $G$, and since $G$ is a monochromatic subgraph of our colouring of $K_n$ it follows that there is a copy of $T$ in $K_n$ with weight $0$.

    Otherwise suppose $G$ is isomorphic to $K_1 \cup K_{n-1}$. Let $x$ be the isolated vertex in $G$. Then the colouring of $K_n$ is such that $K_n - x$ is monochromatically coloured $a$, whilst $x$ has $n - 2$ incident edges coloured $b$ and one incident edge coloured $c$, say to a neighbour $y$. Note that it is possible that $a = b$ or $a = c$, and that $b = c$, but not $a = b = c$.
    
    \begin{figure}[h!]
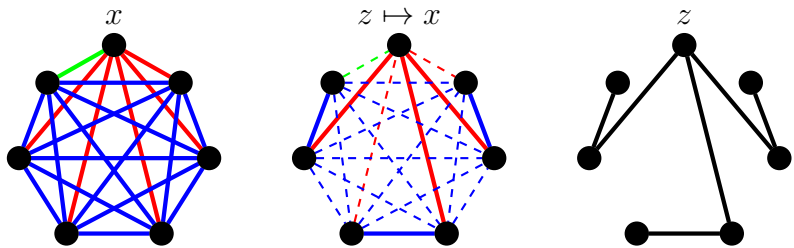

	\label{fig5}
	\ctikzfig{proposition_7_case_1}
	\vspace*{-4mm}
	\caption{(\textit{Case 1}) Illustration of a restrictive colouring of $K_7$ such that the removal of the uniquely coloured edges incident to each vertex is a (blue) $K_1 \cup K_6$, where $x$ is the isolated vertex. The tree $T$ obtained by coalescing three $P_3$ paths at a common vertex $z$ has a zero-sum embedding by associating the vertex $z$ of degree $3$ with the vertex $x$ in $K_n$.}	
	\end{figure}
    
    \textit{\textbf{Case 1:}} Suppose $z$ is a vertex in $T$ of degree $0 \ (\!\!\!\!\mod 3)$. Since $T$ is not isomorphic to $K_{1, n-1}$ then $z$ has at most $n - 2$ neighbours. 
    
    Consider an embedding of $T - z$ in $K_n - x$ such that none of the neighbours of $z$ in $T$ are identified with $y$ in $K_n$ (which is possible since $z$ has at most $n-2$ neighbours). Since $K_n - x$ is monochromatically coloured and $T - z$ has $0 \ (\!\!\!\!\mod 3)$ edges, then this embedding of $T - z$ in $K_n$ has weight $0$. Since the embedding avoids $y$, identifying $z$ with $x$ we can add the $0 \ (\!\!\!\!\mod 3)$ edges connecting $z$ to $T - z$, all of which are coloured the same and hence have total weight $0$. The result is a zero-sum embedding. 
    
    Combining everything together, this gives the first part of our result, namely that if $T$ as in the statement has a vertex of degree $0 \ (\!\!\!\!\mod 3)$ then $T$ has a zero-sum embedding in every restrictive colouring of $K_n$ using $\mathbb{Z}_3$.

    \textit{\textbf{Case 2:}} Otherwise suppose that $T$ does not have a vertex of degree $0 \ (\!\!\!\!\mod 3)$. We have already shown that there exists restrictive colourings of $K_n$ using $\mathbb{Z}_3$ such that $T$ has no zero-sum embedding. Hence consider a restrictive colouring of $K_{n+1}$ using $\mathbb{Z}_3$. Let $G$ be obtained as before; in the case that $G$ is a connected graph with minimum degree at least $n - 1$, by a similar argument as before there is a zero-sum embedding of $T$. Otherwise, the monochromatic graph $G$ is isomorphic to $K_1 \cup K_n$. Then $T$ embeds in the monochromatic component of size $n$, and we are done.
\end{proof}

\begin{remark} \label{rem1}
	Observe that for a tree $T$ with $3 \vert e(T)$ such that no vertex of $T$ has degree $0 \ (\!\!\!\!\mod 3)$, the colouring constructed in the proof of Lemma \ref{r_g_z3_lower} (ii) is a restrictive colouring forcing $R(T, \mathbb{Z}_3) \geq n + 1$.
\end{remark}

\section{Determining $R(T, \mathbb{Z}_3)$ for trees with automorphism switchable pendants}

Let $T$ be a tree and let $v$ be a vertex of $T$. We say that there is an \textit{automorphism switchable pendant (ASP) at $v$} if there exists a non-trivial automorphism $\sigma$ in $\text{Aut}(T-v)$ such that for some vertex $x$ in $T - v$, we have that $x$ is a neighbour of $v$ in $T$ but $\sigma(x)$ is not. 

If $T$ is a rooted tree with root $r$, we say that $T$ has an \textit{ASP from $r$} if there is an ASP at $r$ or at some vertex in a sub-tree from $r$. 

\begin{figure}[h!]
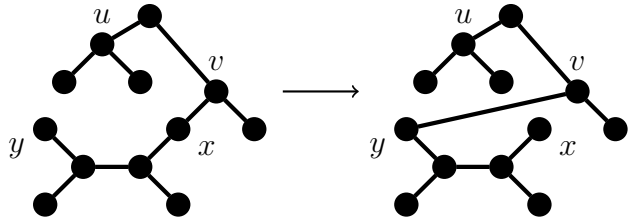

\label{fig1}
\ctikzfig{automorphism_switchable}
\vspace*{-5mm}
\caption{The double-star in $T - v$ with leaf $x$ is an example of an ASP at $v$, namely since the double-star has an automorphism mapping $x$ to another leaf $y$, where $x$ is a neighbour of $v$ in $T$ but $y$ is not. On the other hand, the vertex $u$ has no ASPs.}	
\end{figure}

\begin{lemma} \label{k_weight_lemma}
	Let $n \in \mathbb{N}$ and consider a colouring of $K_n$ using $\mathbb{Z}_3$. Suppose that there exists a vertex $v$ in $K_n$ which has at least two distinctly coloured incident edges. Let $T$ be a tree on $n$ vertices rooted at a vertex $r$.
	
	If there exists an ASP in $T$ from $r$, then there exists two embeddings of $T$ in $K_n$ having distinct weight, with the vertex $r$ in $T$ identified with the vertex $v$ in $K_n$.
\end{lemma}

\begin{proof}
	Let $z$ be the vertex in the rooted tree $T$ such that there is an ASP at $z$. Then there is a neighbour $x$ of $z$ in $T$ such that the sub-tree $T_x$ rooted at $x$ has a non-trivial automorphism $\sigma$ such that $\sigma(x)$ is not a neighbour of $z$ in $T$.
	
	Let $r_x$ be the neighbour of $r$ in $T$ such that the sub-tree $T_{r_x}$ at $r_x$ contains $x$. Note that if $r = z$ then $r_x = x$.
		
	Let $T'$ be the tree obtained by considering the sub-tree $T_{r_x}$ and adding the edge $\{r_x, r\}$. Let $k \in \mathbb{N}$ be the number of vertices in $T'$. Note that $k \geq 3$, namely since $T'$ contains the three distinct vertices $z, x$ and $\sigma(x)$. Let $v_1$ and $v_2$ be the two neighbours of $v$ in $K_n$ such that $\{v, v_1\}$ and $\{v, v_2\}$ are distinctly coloured.
	
	\textit{\textbf{Case 1:}} Suppose that $r = z$. Consider a subset $V'$ of $V(K_n)$ of size $k-1$ such that $v \notin V'$ but $v_1, v_2 \in V'$. Let $H$ be the $(k-1)$-clique induced by $V'$, and consider an embedding of $T_x$ in $H$, such that $x$ is identified with $v_1$ and $\sigma(x)$ is identified with $v_2$. Let $\alpha$ be weight of this embedding. This embedding can be extended to an embedding of $T'$ in $K_n$ in two ways: 
	\begin{enumerate}[i.]
		\item By adding the edge $\{v, v_1\}$ with $v$ identified with $z$, and since $v_1$ is identified with $x$ in $H$, then this is an embedding of $T'$ in $K_n$ of weight $\alpha + a$. 
		\item By adding the edge $\{v, v_2\}$ with $v$ is identified with $z$, and since $v_2$ is identified with $\sigma(x)$ in $H$, where $\sigma$ is an automorphism of $T_x$, then this is an embedding of $T'$ in $K_n$ of weight $\alpha + b$. 
	\end{enumerate}
	
	Note that $\sigma(x)$ not being a neighbour of $z$ ensures that only one of the edges $\{v, v_1\}$ and $\{v, v_2\}$ are required in the embedding.
	
	Since $a \neq b$ then these two embedding of $T'$ in $K_n$ have distinct weight.
	
	\begin{figure}[h!]
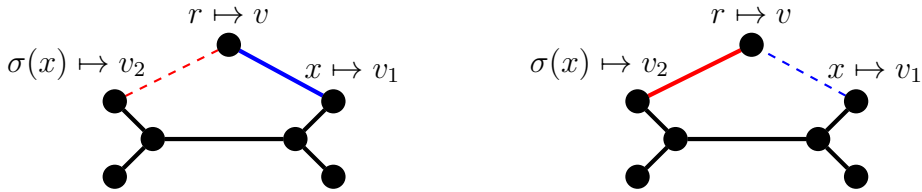

	\label{fig2}
	\ctikzfig{lemma_1_case_1}
	\vspace*{-2mm}
	\caption{(\textit{Case 1}) Illustration of two distinctly weighted embeddings of $T'$ in $K_n$, by considering a fixed embedding of $T_x$ on $V'$ with $x$ and $\sigma(x)$ identified with $v_1$ and $v_2$ respectively, and then adding either the edge $\{v, v_1\}$ or $\{v, v_2\}$. }	
	\end{figure}
	
	\textit{\textbf{Case 2:}} Suppose $r \neq z$ and that for every subset of $V(K_n)$ of size $k-1$ containing $v_1$ and $v_2$ but not $v$, the $(k-1)$-clique in $K_n - v$ induced by such a subset is monochromatic. Let $V'$ be one such subset.

	Consider two embeddings of $T_{r_x}$ in this clique, one with $r_x$ identified with $v_1$ and another with $r_x$ identified with $v_2$. Since the clique is monochromatic, then these two embeddings of $T_{r_x}$ have the same weight. Considering these embeddings in $K_n$, extending each respectively with the edges $\{v, v_1\}$ and $\{v, v_2\}$ where $r$ is identified with $v$, we get two distinctly weighted embeddings of $T'$ in $K_n$.
	
	\begin{figure}[htbp!]
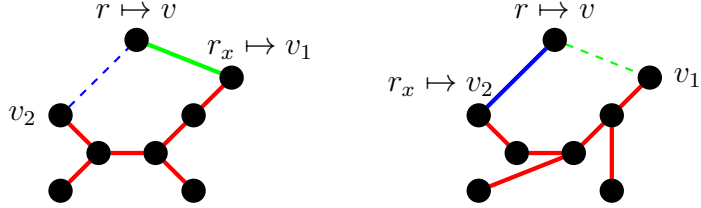

	\label{fig3}
	\ctikzfig{lemma_1_case_2}
	\vspace*{-3mm}
	\caption{(\textit{Case 2}) Illustration of two distinctly weighted embeddings of $T'$ in $K_n$, by considering two monochromatic embeddings of $T_{r_x}$ on $V'$, one with $r_x$ identified with $v_1$ and another with $r_x$ identified with $v_2$,   and then adding the edges $\{v, v_1\}$ and $\{v, v_2\}$ to each embedding respectively.}	
	\end{figure}
		
	\vspace*{-2mm}
		
	\textit{\textbf{Case 3:}} Otherwise, $r \neq z$ and there exists a subset $V'$ of $V(K_n)$ of size $k-1$ such that, in particular, $v \notin V'$ and it induces a non-monochromatic $(k-1)$-clique $H$ in $K_n - v$. Then there exists a vertex $u$ in $H$ which has at least two incident edges coloured distinctly. Let $u_1$ and $u_2$ be the neighbours of $u$ in $H$ such that $\{u, u_1\}$ and $\{u, u_2\}$ have colours $a$ and $b$ respectively, where $a \neq b$. 
	
	Consider an embedding in $H$ of the forest obtained from removing the edge $\{z, x\}$ from $T_{r_x}$, such that $z$ is identified with $u$, $x$ is identified with $u_1$ and $\sigma(x)$ is identified with $u_2$. Let $\alpha$ be the weight of this embedding. We can extend this embedding in two ways: 
	\begin{enumerate}[i.]
		\item By adding the edge $\{u, u_1\}$, and since $u$ is identified with $z$ and $u_1$ is identified with $x$, this is an embedding of $T_{r_x}$ in $H$ of weight $\alpha + a$. 
		\item By adding the edge $\{u, u_2\}$, and since $u$ is identified with $z$ and $u_2$ is identified with $\sigma(x)$ where $\sigma$ is an automorphism of $T_x$, this is an embedding of $T_{r_x}$ in $H$ of weight $\alpha + b$. 
	\end{enumerate}
	
	\vspace*{-2mm}
	
	\begin{figure}[h!]
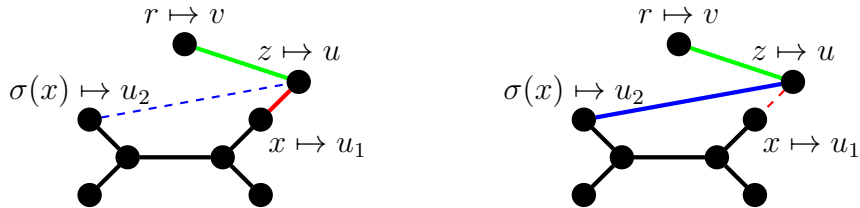

	\label{fig4}
	\ctikzfig{lemma_1_case_3}
	\vspace*{-4mm}
	\caption{(\textit{Case 3}) Illustration of two distinctly weighted embeddings of $T'$ in $K_n$, by considering a fixed embedding of $T_x$ on $V'$ with $z$, $x$ and $\sigma(x)$ identified with $u$, $u_1$ and $u_2$ respectively, and then adding either the edge $\{u, u_1\}$ or $\{u, u_2\}$, along with the edge $\{v, u\}$.}	
	\end{figure}
	
	Note that $\sigma(x)$ not being a neighbour of $z$ ensures that only one of the edges $\{u, u_1\}$ and $\{u, u_2\}$ are required in the embedding.
	
	Since $a \neq b$, these embedding of $T_{r_x}$ in $H$ have distinct weight. In both of these embeddings, $r_x$ is identified with the same vertex in $H$, say $w$ (possibly $u$ in the case that $r_x = z$). Then these embedding can both be extended to embeddings of $T'$ in $K_n$ by adding the edge $\{v, w\}$, such that the resulting embeddings have distinct weight.
	 	
 	In all the cases above, we can extend the two embeddings of $T'$ in the same way in order to embed the remainder of $T$, resulting in two distinctly weighted embeddings of $T$ in $K_n$ with $r$ in $T$ identified with $v$ in $K_n$.
\end{proof}

We are now in a position to state our main results concerning $R(T, \mathbb{Z}_3)$, Theorems \ref{thm1} and \ref{thm2}, before proving each one respectively. 

\begin{theorem} \label{thm1}
	Let $n \in \mathbb{N}$ such that $n  \equiv 1 \ (\!\!\!\!\mod 3)$. Let $T$ be a tree on $n$ vertices. If there exists a non-leaf vertex $v$ in $T$ such that in $T$ rooted at $v$ there is an ASP from $v$, then $R(T, \mathbb{Z}_3) \leq n + 1$ with equality if $T$ does not have a vertex of degree $0 \ (\!\!\!\!\mod 3)$.
\end{theorem}

As an immediate consequence of Theorem \ref{thm1}, we have the following.

\begin{corollary} \label{corr_thm1}
	Every tree $T$ on $n \equiv 1 \ (\!\!\!\!\mod 3)$ vertices in which there is a leaf adjacent to a vertex of degree 2 has $R(T, \mathbb{Z}_3) \leq n + 1$, and if there is no vertex of degree $0 \ (\!\!\!\!\mod 3)$ then $R(T, \mathbb{Z}_3) = n +1$.
\end{corollary}

Two ASPs in a tree $T$ rooted at $r$ are said to be \textit{separated} if they are either ASPs at $r$, or ASPs at two vertices in distinct sub-trees from $r$. 

\begin{theorem} \label{thm2}
	Let $n \in \mathbb{N}$ such that $n  \equiv 1 \ (\!\!\!\!\mod 3)$. Let $T$ be a tree on $n$ vertices. If there exists a vertex $v$ in $T$ such that considering $T$ rooted at $v$ there are two separated ASPs from $v$, then \[
        R(T, \mathbb{Z}_3) = \begin{cases}
            n & \text{if there exists a vertex of degree} \ 0 \ (\!\!\!\!\!\!\mod 3), \\
            n + 1 & \text{otherwise}.
        \end{cases}
    \]
\end{theorem}

These two theorems together give sharp bounds on $R(T, \mathbb{Z}_3)$ for infinite classes of trees. As a consequence of Theorem \ref{thm1}, we have that given $k_1 \geq 2$ and $k_2 \in \{2, 3\}$ such that $2k_1 + k_2 \equiv 1 \ (\!\!\!\!\mod 3)$, if $T$ is two stars $K_{1, k_1}$ joined by a path $P_{k_2}$ between their centres, then $R(T, \mathbb{Z}_3) \leq n + 1$. 

Likewise, as a consequence of Theorem \ref{thm2}, for $k_2 \geq 4$, then $R(T, \mathbb{Z}_3) = n$ if $k_1 \equiv 2 \ (\!\!\!\!\mod 3)$ and $R(T, \mathbb{Z}_3) = n + 1$ otherwise. This is exemplified in Figure \ref{fig6}.

\begin{figure}[h!]
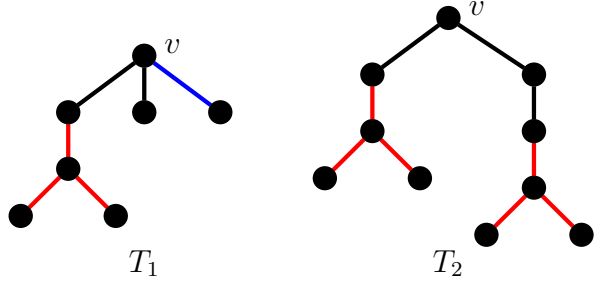

	\ctikzfig{thm_19_20}
	\vspace*{-4mm}
	\caption{Trees $T_1$ and $T_2$ as in Theorems \ref{thm1} and \ref{thm2} respectively. The ASPs in the subtrees from $v$ are highlighted in red, whilst the leaf required in Theorem \ref{thm1} is highlighted in blue.}
	\label{fig6}
	\end{figure}

\begin{proof}[\textbf{Proof of Theorem \ref{thm1}}]
	Firstly note that since $v$ is not a leaf, then there must be some sub-tree from $v$, distinct from the one containing the ASP, which contains a leaf. In particular, we can choose the vertex $v$ such that we have a leaf incident to $v$ and the ASP in the remainder.

	The smallest such tree must have at least $7$ vertices and is not isomorphic to a star rooted at $v$. Consider an edge-colouring of $K_{n+1}$ using $\mathbb{Z}_3$. If this colouring is restrictive, then by Proposition \ref{zero_sum_restrictive_colouring_1} it immediately follows that $R(T, \mathbb{Z}_3) \leq n + 1$.
	
	Otherwise suppose that this edge-colouring is not restrictive. Then there exists some vertex $u$ and at least two pairs of vertices, say $(w, x)$ and $(y, z)$, such that the edges $\{u, w\}$ and $\{u, x\}$ are coloured distinctly, and similarly for $\{u, y\}$ and $\{u, z\}$.
	
	Let $T'$ be the sub-tree of $T$ obtained by considering the component in $T - v$ containing the ASP and adding $v$. Let $l + 1$ be the number vertices in $T'$. Consider a partition of $V(K_{n+1}) \backslash \{y, z\}$ intro two sets $V'$ and $V_0$ of sizes $l$ and $n - l - 1$ respectively, such that $u \in V_0$ and $w, x \in V'$. 
	
	Let $G'$ be the coloured subgraph induced by $V' \cup \{u\}$, isomorphic to $K_{l+1}$. Since in $G'$ there are two distinctly coloured edges incident to $u$, by Lemma \ref{k_weight_lemma} there are two embeddings of $T'$ in $G'$, with $v$ identified with $u$, having distinct weight. Moreover, for a chosen leaf incident to $v$ in $T$, there are two distinctly weighted ways to embed it in $K_{n + 1}$, namely using the differently weighted edges $\{u, y\}$ and $\{u, z\}$. 
	
	The remainder of the tree, having $n - l - 1$ distinct vertices in total, may be embedded in the coloured subgraph $G_0$ induced by $V_0$, isomorphic to $K_{n - l - 1}$, with $v$ in the remainder identified with $u$ in $G_0$. Let $w_0$ be the weight of this embedding of $T_0$. Recall that if $A$ and $B$ are two subsets of $\mathbb{Z}_3$ having size $2$ each, by the simplest form of the Cauchy-Davenport Theorem we have that $|A + B| = 3$ \textit{i.e.} $A + B = \mathbb{Z}_3$. Hence, we can find a copy of $T'$ in $G'$ and a copy of the chosen leaf incident to $v$ such that the sum of their weights is $-w_0$. 
	
	Consequently, there is a zero-sum copy of $T$ in $K_{n+1}$ and the result follows.
\end{proof}

\begin{proof}[\textbf{Proof of Theorem \ref{thm2}}]
	Consider an edge-colouring of $K_n$ by $\mathbb{Z}_3$ which is not restrictive. Then there exists some vertex $u$ and at least two pairs of vertices, say $(w, x)$ and $(y, z)$, such that the edges $\{u, w\}$ and $\{u, x\}$ are coloured distinctly, and similarly for $\{u, y\}$ and $\{u, z\}$. 
	
	Now, let $T_L$ and $T_R$ be two trees obtained from considering any two distinct sub-trees from $v$ containing an ASP and adding the unique edge to $v$ in each. Let $l + 1$ and $r + 1$ be the number of vertices in $T_L$ and $T_R$ respectively.
	
	Consider a partition of $V(K_n)$ into three sets $V_L$, $V_R$ and $V_0$ of sizes $l$, $r$ and $n - l - r$, respectively, such that $u \in V_0$. Furthermore, let this partition be such that $w$ and $x$ are in $V_L$, whilst $y$ and $z$ are in $V_R$. Let $G_L$ be the coloured subgraph induced by $V_L \cup \{u\}$, isomorphic to $K_{l+1}$. Since in $G_L$ there are two distinctly coloured edges incident to $u$, by Lemma \ref{k_weight_lemma} there are two embeddings of $T_L$ in $G_L$, with $v$ identified with $u$, having distinct weight. Similarly, let $G_R$ be the coloured subgraph induced by $V_R \cup \{u\}$, isomorphic to $K_{r+1}$. Then there are two embeddings of $T_R$ in $G_R$, with $v$ identified with $u$, having distinct weight. 
	
	The remainder of the tree $T$, having $n - l - r$ distinct vertices in total, may be embedded in the coloured subgraph $G_0$ induced by $V_0$, isomorphic to $K_{n - l - r}$, with $v$ in the remainder identified with $u$ in $G_0$. Let $w_0$ be the weight of this embedding of $T_0$. Recall that if $A$ and $B$ are two subsets of $\mathbb{Z}_3$ having size $2$ each, by the simplest form of the Cauchy-Davenport Theorem we have that $|A + B| = 3$ \textit{i.e.} $A + B = \mathbb{Z}_3$. Hence, we can find a copy of $T_L$ in $G_L$ and a copy of $T_R$ in $G_R$ such that the sum of their weights is $-w_0$. 

    Otherwise suppose that our colouring of $K_n$ is a restrictive colouring using $\mathbb{Z}_3$. Note that by the condition that there are at least two ASPs and $n \equiv 1 \ (\!\!\!\!\mod 3)$, then $T$ is not isomorphic to $K_{1, n-1}$ and $n \geq 7$. 
    
    By Proposition \ref{zero_sum_restrictive_colouring_1}, if $T$ has a vertex of degree $0 \ (\!\!\!\!\mod 3)$ then there is a zero-sum embedding of $T$ and combined with the previous case, we have that $R(T, \mathbb{Z}_3) = n$. Else by Lemma \ref{r_g_z3_lower} and Remark \ref{rem1}, if $T$ does not have a vertex of degree $0 \ (\!\!\!\!\mod 3)$, there exists such colourings such that there is no zero-sum embedding of $T$, however by Proposition \ref{zero_sum_restrictive_colouring_1} every such colouring of $K_{n+1}$ has a zero-sum embedding of $T$. Combined with the previous case, we have that $R(T, \mathbb{Z}_3) = n + 1$.
 \end{proof}
 
 \section{Concluding remarks}
 
 We have seen that for every $2$-good graph $G$ on $n$ vertices with $3 \vert e(G)$, then $R(G, \mathbb{Z}_3) \leq n + 2$, and in particular for a tree $T$ not isomorphic to a star it follows that $R(T, \mathbb{Z}_3) \leq n + 2$, matching the upper-bound implied by Conjecture \ref{conj_tree_gen} for $k = 3$. More generally, for $k \geq 3$ and a tree $T$ such that $k \vert e(T)$, Conjecture \ref{conj_tree_gen} implies that $R(T, \mathbb{Z}_k) \leq n + k - 1$.  
 
 This motivates the following problem.
 
 \begin{problem}
 	Let $k \in \mathbb{N}, k \geq 3$. A graph $G$ is a said to be a $(k-1)$-good graph if it has at least $k - 1$ vertices of degree $1$ that are pairwise a distance at least three apart. 
 	
 	If $G$ is a $(k-1)$-good graph such that $k \vert e(G)$, is $R(G, \mathbb{Z}_k) \leq n + k - 1$?
 \end{problem}

\begin{theorem} \label{r_g_zp_upper}
Let $p$ be a prime such that $p \geq 3$. If $G$ be a $(p - 1)$-good graph on $n$ vertices with $p \vert e(G)$, then $R(G, \mathbb{Z}_p) \leq n + 3p - 6$.
\end{theorem}

\begin{proof}
    Let $m = n + 3p - 6$ and consider an edge-colouring $f$ of $K_m$ using $\mathbb{Z}_p$. We proceed by cases on the number of vertex disjoint copies of $K_{1, 2}$ in $K_m$ such that the two edges of each copy are coloured different under $f$.

    \textit{\textbf{Case 1:}} Suppose that there are at most $p - 2$ such vertex disjoint copies of $K_{1, 2}$ with the desired property. Removing them from $K_m$, we are left with at least an $n$-clique, since we removed at most $3(p-2)$ vertices. 

    In this $n$-clique, with the colouring $f$ restricted to it, all edges must be coloured the same, as otherwise there exists a vertex $v$ with two neighbours $u, w$ in this clique such that $f(\{v, u\}) \neq f(\{v, w\})$, and hence we have another copy of $K_{1, 2}$ with the desired property that is disjoint from all other $(p - 2)$ copies --- a contradiction. 
    
    Hence we have a monochromatic $n$-clique in $K_m$ coloured using $f$, and any embedding of $G$ in this clique yields a zero-sum copy of $G$ since $p \vert e(G)$.

    \textit{\textbf{Case 2:}} Otherwise suppose that there are at least $p - 1$ vertex disjoint copies of $K_{1,2}$ with the desired property. Consider $p - 1$ of these copies, and for all $1 \leq j \leq p - 1$, let $v_j$ be the centre and $v_{j, 1}, v_{j, 2}$ be the leaves of copy $j$. Under $f$, we have that $f(\{v_j, v_{j, 1}\}) \neq f(\{v_j, v_{j, 2}\})$ for all $1 \leq j \leq p - 1$.

    Since $G$ is a $(p - 1)$-good graph, there are $p - 1$ distinct vertices $w_1, \dots, w_{p - 1}$ adjacent to a leaf, say $l_1, \dots, l_{p -1}$ respectively. Let $G^\ast = G \backslash \{l_1, \dots, l_{p - 1}\}$. Observe that $|V(G^*)| = n - p + 1$.

    Consider $K_m$ with $v_{j, 1}, v_{j, 2}$ removed for all $1 \leq j \leq p - 1$; then we are left with a clique on $m - 2(p - 1) = n + p - 4$. Since $n + p - 4 \geq n - p + 1 = |V(G^*)|$ for all $p \geq 3$, we can embed $G^*$ in this clique such that every $w_j$ in $G^*$ is mapped onto $v_j$ in the clique for all $1 \leq j \leq p - 1$. Let $w(G^*)$ be the weight of this embedding.

    Extending this to an embedding of $G$ in $K_m$, for all $1 \leq j \leq p - 1$, we have at least two distinctly coloured edges with which we can add the leaf $l_j$, namely the edges $\{v_j, v_{j, 1}\}$ and $\{v_j, v_{j, 2}\}$. Let $A_j = \{f(\{v_j, v_{j, 1}\}), \{v_j, v_{j, 2}\}\}$ for all $1 \leq j \leq p - 1$ and note that $|A_j| = 2$. By the general form of the Cauchy-Davenport Theorem, we have that $|A_1 + A_2 + \dots + A_{p - 1}| \geq \min\{p, 2(p - 1) - (p - 1) + 1\} = p$. 
    
    Then we can find a choice of embedding for each of the $p - 1$ leaves such that the weight of the chosen edges is $-w(G^*)$, and extending the embedding of $G^*$ by these edges yields an embedding of $G$ with weight $w(G^*) - w(G^*) = 0$. The result follows.
\end{proof}

Since the submission of this paper, Conjecture \ref{conj_tree} was (slightly corrected and) fully proved in \cite{AlvaradoColucciParente}.
 
\bibliographystyle{plain}
\bibliography{zs_trees_ref}

\end{document}